\documentclass[twoside,11pt]{article}

\usepackage{jmlr2e}

\usepackage{fullpage,amsmath,amssymb,color}
\usepackage{algorithm}
\usepackage{enumerate}
\usepackage{algorithmic}

\usepackage{graphicx}
\usepackage{epstopdf}
\usepackage{subfigure}
\graphicspath{{figures/}}
\usepackage{graphicx}   
\usepackage{multirow}
\usepackage{amsmath}
\usepackage{amsfonts}
\usepackage{latexsym}
\usepackage{verbatim}
\usepackage{amsmath}    
\usepackage{amsmath, amssymb, algorithm, algorithmic}
\usepackage[table]{xcolor}

\usepackage{multirow}
\newtheorem{assumption}{Assumption}
\usepackage [pagewise,displaymath, mathlines]{lineno}
\usepackage{amssymb}
\usepackage{epstopdf}
\usepackage{amsmath}
\usepackage{epsf}
\usepackage{graphicx}
\usepackage{mathdots}
\newcommand{\beq}{\begin{equation}}
\newcommand{\eeq}{\end{equation}}
\newcommand{\beqa}{\begin{eqnarray}}
\newcommand{\eeqa}{\end{eqnarray}}
\newcommand{\beqas}{\begin{eqnarray*}}
\newcommand{\eeqas}{\end{eqnarray*}}
\newcommand{\bi}{\begin{itemize}}
\newcommand{\ei}{\end{itemize}}
\newcommand{\ba}{\begin{array}}
\newcommand{\ea}{\end{array}}

\newcommand{\nn}{\nonumber}

\def\eqnok#1{(\ref{#1})}

\def\vgap{\vspace*{.1in}}

\def\exp{{\rm exp}}

\newcommand{\bbr}{\Bbb{R}}

\def\cS{{\cal S}}

\def\lb{{\rm lb}}

\ShortHeadings{}{}
\firstpageno{1}

\begin{document}

\title{Second-Order Methods with Cubic Regularization \\
Under Inexact Information}
\author{\name Saeed Ghadimi\thanks{Corresponding author} \email sghadimi@princeton.edu \\
       \addr Department of Operations Research and Financial Engineering\\
       Princeton University\\
       Princeton, NJ 08544, USA
       \AND
       \name Han Liu \email hanliu@princeton.edu \\
       \addr Department of Operations Research and Financial Engineering\\
       Princeton University\\
       Princeton, NJ 08544, USA
       \AND
       \name Tong Zhang \email tzhang@tencent.com \\
       \addr Tencent AI Lab\\
       Shenzhen, China}

\editor{
\\
\\
{\color{white}gkkfkkfffffffffffffffffffffffjsjfkskfskgksgskgwsholam}July 6, 2017}

\maketitle

\date{July 6, 2017}
\begin{abstract}
In this paper, we generalize (accelerated) Newton's method with cubic regularization under inexact second-order information for (strongly) convex optimization problems. Under mild assumptions, we provide global rate of convergence of these methods and show the explicit dependence of the rate of convergence on the problem parameters.
While the complexity bounds of our presented algorithms are theoretically worse than those of their exact counterparts, they are at least as good as those of the optimal first-order methods. Our numerical experiments also show that using inexact Hessians can significantly speed up the algorithms in practice.
\end{abstract}

\noindent {\bf keywords}
convex optimization, strongly convex problems, inexact Newton method, Hessian approximation, accelerated methods.

\setcounter{equation}{0}

\section{Introduction}
\label{sec_intro}

In this paper, we consider the following classical optimization problem:
\begin{equation}
\min_{x \in \bbr^n} f(x),
\label{main_prob}
\end{equation}
where $f$ is a twice continuously differentiable strongly convex function with parameter $\lambda \ge 0$ i.e.,
\beq \label{strong_cvx}
f(\alpha x+(1-\alpha)x') \le \alpha f(x) + (1-\alpha) f(x') - \frac{\alpha (1-\alpha) \lambda}{2} \|x-x'\|^2 \ \ \forall x,x' \in \bbr^n, \ \ \forall \alpha \in [0,1],
\eeq
and $\lambda =0$ simply corresponds to the case where $f$ is only convex. Moreover, we may assume that Hessian of $f$ is $\gamma$-Lipschitz continuous with respect to the spectral norm i.e.,
\beq \label{lips2}
\|\nabla^2 f(x)- \nabla^2 f(x')\| \leq \gamma \|x-x'\| \ \ \forall x,x' \in \bbr^n.
\eeq
Despite the long history of the classic Newton's method, second-order methods have not become as popular as first-order ones for solving optimization problems. The main reason is that each iteration of this type of methods is expensive due to the computation of Hessian matrix of the objective function or its inverse. This computation can be prohibitive for large-scale problems. In the past few years, second-order methods reducing this iteration computational cost have been getting much more interest. For example, \cite{AgrBuHa16} presented an algebraic technique to compute an unbiased estimator of Hessian inverse which only linearly depends on the dimension of the problem in the case of linear models in machine learning problems. Using this technique, the authors present a Newton-type method and obtain linear rate of convergence for minimizing strongly convex functions given in the form of regularized finite sum of functions. The idea of using subsmaples to compute Hessian approximations is another approach which has been used to reduce the computational cost of each iteration \citep[see e.g.,][]{ByHaNoSi16,BoByNo16,RooMah16-1,RooMah16-2}. In all of these studies, different variants of Newton's method have been proposed which include minimizing a second-order Taylor's approximation of the objective function with or without a quadratic term at each iteration. More specifically, to update the current solution $x_k$, these methods iteratively solve subproblems of the form
\beq \label{subproblem_quad}
\min_x \left\{f(x_k) +\langle \nabla f(x_k), x -
   x_k\rangle + \frac{\eta_1}{2} \langle H_k(x-x_k), x-x_k \rangle +\frac{\eta_2}{2} \|x-x_k\|^2\right\},
\eeq
where $H_k$ is an approximation of Hessian at $x_k$, $\eta_1>0$, and $\eta_2 \ge 0$ are algorithmic parameters. Note that if $H_k=\nabla^2 f(x_k)$, $\eta_1=1$ and $\eta_2 =0$, then the above subproblem reduces to the classic Newton's step formula.

On the other hand, another type of Newton methods consider a second-order Taylor's expansion with a cubic regularized term. In particular, they iteratively solve subproblems of the form
\beq \label{subproblem_cubic}
\min_x \left\{f(x_k) +\langle \nabla f(x_k), x -
  x_k \rangle + \frac{1}{2} \langle H_k(x-x_k), x-x_k \rangle+ \frac{\eta}{6} \|x- x_k\|^3\right\}
\eeq
for some $\eta>0$. Note that if $H_k=\nabla^2 f(x_k)$ and $f$ satisfies \eqnok{lips2}, then the function minimized above will be an upper bound on $f(x_k)$ for any $\eta \ge \gamma$. This property plays an important role in providing convergence analysis of this kind of methods. A variant of Newton's method using this form of subproblems seems to be first introduced in \citep{Gri81} and showed to be convergent to the second-order critical point of the problem while the subproblems are not necessarily solved exactly. More recently, \cite{NestPoly06-1} also studied a similar method and established its rate of convergence for different class of (strongly) convex and nonconvex problems. An accelerated variant of this algorithm is also proposed by \cite{Nest08-2} which can speed up the rates of convergence for (strongly) convex problems. While these works require $H_k$ to be the exact Hessian of the objective function, \cite{CarGouToi11-1} generalized this approach and presented a (second-order) adaptive cubic regularized (ARC) method, by using Hessian approximations and solving subproblems approximately \citep[see also][]{CarGouToi11-2,CarGouToi12-2}. The authors provide convergence analysis of their methods when applied to nonconvex and (strongly) convex problems by making some assumptions on the Hessian approximations. Efficient subroutines also have been proposed for solving the subproblems in the form of \eqnok{subproblem_cubic}, which make the above-mentioned cubic regularized variants of Newton's method more practically (see Section~\ref{sec_impl} for more details).

\subsection{Our contribution}
Our main contribution in this paper consists of the following aspects.
First, we generalize the Newton's method with cubic regularization presented in \citep{NestPoly06-1} under inexact second-order information
and establish its rate of convergence for the class of (strongly) convex optimization problems. In particular, we show that under mild assumptions on Hessian approximations, the total number of iterations performed by this algorithm to find an $\epsilon$-solution $\bar x $ of problem \eqnok{main_prob}, when $f$ is convex, such that $f(\bar x)-f(x_*) \le \epsilon$, is bounded by
\[
{\cal O}(1)\left\{\frac{\mu_u R^2}{\epsilon}+\sqrt{\frac{\gamma R^3}{\epsilon}}\right\}
\]
for some upper bounds $\mu_u$ and $R$ on the errors in Hessian approximations and the distance between the generated solutions and an optimal solution $x_*$, respectively. Note that the first term of the above bound is in the same order of the one obtained by the gradient descent method when the gradient of $f$ is Lipschitz continuous and is also achieved by the ARC method presented in \citep{CarGouToi12-2}. However, the assumptions on Hessian approximations and analysis of the ARC method are different and its complexity bound does not separate the affect of using inexact Hessians. More specifically, when better Hessian approximations are used and so $\mu_u$ vanishes, the above complexity bound for our method is reduced to the bound obtained in \citep{NestPoly06-1,CarGouToi12-2} for convex problems when exact Hessians are used. Moreover, when $f$ is strongly convex with parameter $\lambda>0$, an $\epsilon$-solution is obtained in at most
\[
{\cal O}(1)\max\left\{1,\frac{\mu_u}{\lambda},\sqrt{\frac{\gamma R}{\lambda}}\right\}\ln\left(\frac{f(x_0)-f(x_*)}{\epsilon}\right),
\]
number of iterations, which is also in the same order of the one obtained by the gradient descent method and the ARC method in \citep{CarGouToi12-2} for smooth problems.

Second, while the aforementioned complexity bounds are worse than those obtained by optimal first-order methods, we present accelerated variants of our algorithm as well. Indeed, we also generalize the accelerated Newton's method with cubic regularization in \citep{Nest08-2} when only Hessian approximations are available. We show that, when $f$ is convex, this accelerated algorithm exhibits an improved complexity bound of
\[
{\cal O}(1)\left\{\sqrt{\frac{\mu_u \|x_0-x_*\|^2}{\epsilon}}+\left(\frac{\gamma \|x_0-x_*\|^3}{\epsilon}\right)^\frac13\right\},
\]
To the best of our knowledge, there is no such a global complexity bound for an algorithm using only inexact second-order information when applied to convex optimization problems. When exact Hessians are used i.e., $\mu_u=0$, this bound is reduced to the one obtained in \citep{Nest08-2}. Furthermore, if we use restart this accelerated algorithm in a multi-stage framework, it improves the aforementioned complexity bound for strongly convex problems to
\[
{\cal O}(1)\left[\left(\frac{\gamma R}{\lambda}\right)^\frac13+\sqrt{\frac{\mu_u}{\lambda}} \ln\left(\frac{f(x_0)-f(x_*)}{\epsilon}\right)\right].
\]
Note that the aforementioned complexity bounds for the accelerated algorithms can be better than those of optimal first-order methods for the problems with $L$-smooth gradients, when relatively good Hessian approximations are computable i.e., $\mu \ll L$. Also, when $\mu=0$, the above bound is reduced to the number of iterations required by the accelerated algorithm in \citep{Nest08-2}, using exact Hessians, to enter the region of quadratic convergence.

Finally, we show that when the objective function is given as a finite sum of functions, we can properly subsample the exact Hessian to satisfy the assumptions on Hessian approximations with high probability. In this case, the above complexity bounds are viewed as high probability results for obtaining an $\epsilon$-solution of problem \eqnok{main_prob}. Moreover, we show that our methods can be faster than their exact counterparts and optimal first-order methods in practice (see Section~\ref{sec_numerical} for more details).

\subsection{Other related works}
We now review some other related works in more details. \cite{ByHaNoSi16} presented a stochastic variant of the L-BFGS algorithm in which subsampled Hessians are computed at times and showed that it possess sublinear rate of convergence when applied to strongly convex stochastic optimization problems. \cite{ErdMon15} presented a variant of Newton's method by using the idea of subsampling for minimizing strongly convex functions given in the form of finite sum of functions. This method exhibits quadratic rate of convergence at the beginning and then linear rate close to the minimizer.
\cite{RooMah16-1,RooMah16-2} use the idea of Hessian subsampling to present Newton-type methods and establish their local and global convergence properties with high probability. They extended the convergence analysis when the gradients are also subsampled or subproblems are solved inexactly. Non-uniform subsampling for computing Hessian approximations is also studied in \citep{XuYaRoReMa16}.
\cite{BoByNo16} proposed a variant of Newton's method for minimizing strongly convex problems in which both gradients and Hessians are subsampled. The authors show that if the subsample sizes for gradients and Hessians increase faster than a geometric rate and at any rate, respectively, then this method converges superlinearly under a boundedness assumption on moments of the generated trajectory.

All of the above studies include subproblems in the form of \eqnok{subproblem_quad} within framework of their proposed methods. Very recently,
a couple of works presented Newton-type methods including cubic regularized subproblems in the form of \eqnok{subproblem_quad} for finding local minimum of nonconvex optimization problems \citep{AgZhBuHaMa16,CarDuc16}. While these methods use exact Hessian in their framework, a few sub-routines have been proposed to find approximate solutions of the subproblems (see also Section~\ref{sec_impl}).

\subsection{Notation}
We use $\|\cdot\|$ for the Euclidean norm of a vector or spectral norm of a matrix. $\nabla f(x)$ and $\nabla^2 f(x)$ are used to denote first- and second-order derivatives of twice differentiable function $f(x)$. Also, $h'(x)$ and $\partial h(x)$ denote a subgradient and the subdiffernetial set of convex function $h$ at $x$, respectively. We use $B_g(x;x')$ to denote Bergman divergence of a convex function $g$ given by
\beq \label{def_breg}
B_g(x;x')=g(x)-g(x') - \langle \nabla g(x'), x-x' \rangle.
\eeq
A solution $\bar x \in \bbr^n$ is said an $\epsilon$-solution of problem \eqnok{main_prob} for some $\epsilon>0$ if $f(\bar x) - f(x_*) \le \epsilon$, where $x_*$ is an optimal solution for \eqnok{main_prob}. ${\cal O}(1)$ is also used to denote a small positive constant number. Diameter of a bounded set $\cS$ is also defined as $diam(\cS) = \max_{x,x' \in \cS} \|x-x'\|$.

\setcounter{equation}{0}
\section{Inexact Newton method with cubic regularization} \label{sec_inexact}

Our goal in this section is to generalize the Newton's method with cubic regularization proposed by \cite{NestPoly06-1} under the inexact second-order information. We provide global rate of convergence this method when applied to (strongly) convex optimization problems.
Below, we present an inexact Newton method using only approximation of Hessian of the objective function.

\begin{algorithm} [H]
	\caption{The Inexact Newton method with cubic regularization (INCR)}
	\label{alg_INCR}
	\begin{algorithmic}

\STATE Input:
$x_0 \in \bbr^n$ and nonnegative sequence $\{\eta_t\}_{t \ge 0}$.

\STATE 0. Set $t=0$.
\STATE 1. Compute an approximation $H_t$ of the $\nabla^2 f(x_t)$.
\STATE 2. Compute
\beq \label{def_xt_cub}
x_{t+1} = \arg\min_{x \in \bbr^n} \tilde{f}_{\eta_t}(x;x_t),
\eeq
where
\beq \label{cub_schem}
\tilde{f}_{\eta_t}(x;x_t)= f(x_{t}) +\langle \nabla f(x_{t}), x -
  x_{t}\rangle + \frac{1}{2} \langle H_t(x-x_{t}), x-x_{t} \rangle + \frac{\eta_t}{6} \|x-x_t\|^3.
\eeq

\STATE 3. Set $t \leftarrow t+1$ and go to step 1.
	\end{algorithmic}
\end{algorithm}

We now add a few remarks about the above algorithm. First, note that at each iteration of the above algorithm, we need to approximate Hessian of $f$. This can be done in different ways. We will discuss this later.
We now make the following assumption on the error of this approximation.
\begin{assumption}\label{H_assumption}
Let $\{x_t\}_{t \ge 0}$ be generated by Algorithm~\ref{alg_INCR}. Then, we have
\beq \label{H_assump}
0 \preceq H_t - \nabla^2 f(x_t) \preceq \mu_t I \ \ \forall t \ge 0,
\eeq
where $0 \leq \mu_t \leq \mu_u$ for some $\mu_u \ge 0$.
\end{assumption}
Note that if gradient of $\nabla f(\cdot)$ is Lipschitz continuous with constant $L>0$, then we have $\|\nabla^2 f(\cdot) \| \le L$. Hence, one can simply choose
\beq \label{H_defL}
H_t = LI \ \ \forall t \ge 1
\eeq
to satisfy the above assumption with $\mu_t = \mu = L$. However, we are not exploiting second-order information of the objective function in this case and numerical performance of the algorithm may not be better than first-order methods (see Section~\ref{sec_numerical}). We will discuss a better choice of Hessian approximation in Section~\ref{sec_impl} which satisfies Assumption~\ref{H_assumption} with high probability, when the objective function is given as a finite sum. Later in this section, we will also present a relaxed variant of the this assumption.

Second, note the the Algorithm~\ref{alg_INCR} differs from the Newton's method with cubic regularization proposed by \cite{NestPoly06-1} in two aspects, namely, using inexact Hessians and allowing adaptive cubic regularization parameter in the subproblems.
The former enables us to apply this algorithm for solving large scale problems in which exact Hessians are computationally expensive and the latter allows us to require a relaxed assumption on Hessian approximations (see Assumption~\ref{H_assumption_mod}).

Third, we need to solve subproblem \eqnok{cub_schem} every iteration. It is shown in \citep{NestPoly06-1} that solving this kind of
subproblems exactly, requires decomposition of matrices $H_t$ which can be prohibitive for large scale problems.
In Section~\ref{sec_impl}, we discuss different approaches to cheaply solve these subproblems.
Finally, to establish convergence of Algorithm~\ref{alg_INCR}, we need to properly choose the cubic regularization parameter $\eta_t$.
We address this issue after presenting the following technical result which plays the key role in our convergence analysis.

\begin{lemma}\label{main_recur}
Let $\{x_t\}_{t \ge 0}$ be generated by Algorithm~\ref{alg_INCR} and Hessian of $f$ be $\gamma$-Lipschitz continuous. Then, for any $x \in \bbr^n$, we have
\beq \label{inex_rec}
f(x_{t+1}) \le f(x)+\Delta_{H_t}- \frac{\eta_t}{12}\|x-x_{t+1} \|^3  + \frac{\gamma+\eta_t}{6}\|x-x_t \|^3+\frac{\gamma-\eta_t}{6} \|x_{t+1}-x_t\|^3,
\eeq
where
\beq\label{delta_H}
\Delta_{H_t} = -\frac12 [\langle H_t(x-x_{t+1}), x-x_{t+1}\rangle+ \langle (H_t-\nabla^2 f(x_t))(x_{t+1}-x_t),x _{t+1}-x_t\rangle + \langle (\nabla^2 f(x_t)-H_t)(x-x_t),x-x_t\rangle ].
\eeq
\end{lemma}

\begin{proof}
Note that by \eqnok{lips2} and the mean-value theorem, for any $x,x' \in \bbr^n$, we can easily show that
\[
- \frac{\gamma}{6} \|x-x'\|^3 \leq
f(x)- \left[f(x') + \langle \nabla f(x'), x-x'\rangle + \frac{1}{2} \langle \nabla^2
f(x')(x-x'), x-x'\rangle \right] \leq \frac{\gamma}{6} \|x-x'\|^3.
\]
Letting $x'=x_t$ in the above relation and noting \eqnok{cub_schem}, we have
\beq \label{f-lower-upper-bound}
- \frac{\gamma}{6} \|x-x'\|^3 \leq f(x)-\left[\tilde{f}_{\eta_t}(x;x_t)+ \frac{1}{2} \langle (\nabla^2
f(x_t)-H_t)(x-x_t), x-x_t\rangle - \frac{\eta_t}{6} \|x-x'\|^3 \right]
\leq \frac{\gamma}{6} \|x-x'\|^3
\eeq
Noting that $x_{t+1}$ achieves the minimum of $\tilde{f}_{\eta_t}(x;x_t)$ due to \eqnok{def_xt_cub}, we have
\begin{equation}
\nabla \tilde{f}_{\eta_t}(x_{t+1};x_t)=0, \label{it-opt}
\end{equation}
which together with the definition of Bregman divergence $B_g(x;x_{t+1})$ with $g(x)=\frac{\eta_t}{6} \|x-x_t\|^3$, imply that
\begin{align}
\tilde{f}_{\eta_t} (x;x_t) - \tilde{f}_{\eta_t}(x_{t+1};x_t) = & \tilde{f}_{\eta_t}(x;x_t) -
\tilde{f}_{\eta_t}(x_{t+1};x_t)  - \langle \nabla \tilde{f}_{\eta_t}(x_{t+1};x_t), x-x_{t+1}\rangle \nonumber \\
= & \frac12 \langle H_t(x-x_{t+1}), x-x_{t+1}\rangle + B_g(x;x_{t+1}), \label{f-tilde-bound}
\end{align}
Therefore, we obtain
\begin{align}
f(x_{t+1}) &\leq \tilde{f}_\eta(x_{t+1};x_t)+\frac{\gamma-\eta_t}{6} \|x_{t+1}-x_t\|^3+ \frac{1}{2} \langle (\nabla^2
f(x_t)-H_t)(x_{t+1}-x_t), x_{t+1}-x_t\rangle  \nn \\
&\leq \tilde{f}_\eta(x;x_t)-\frac12 \langle H_t(x-x_{t+1}), x-x_{t+1}\rangle + \frac{1}{2} \langle (\nabla^2 f(x_t)-H_t)(x_{t+1}-x_t), x_{t+1}-x_t\rangle \nn \\
&- B_g(x;x_{t+1}) +\frac{\gamma-\eta_t}{6} \|x_{t+1}-x_t\|^3 \nn \\
&\le f(x) -\frac12 \langle H_t(x-x_{t+1}), x-x_{t+1}\rangle+ \frac{1}{2} \langle (\nabla^2 f(x_t)-H_t)(x_{t+1}-x_t),x _{t+1}-x_t\rangle \nn \\
&- \frac{1}{2} \langle (\nabla^2 f(x_t)-H_t)(x-x_t),x-x_t\rangle -\frac{\eta_t}{12}\|x-x_{t+1} \|^3 + \frac{\gamma+\eta_t}{6}\|x-x_t \|^3+\frac{\gamma-\eta_t}{6} \|x_{t+1}-x_t\|^3,
\end{align}
where the first inequality follows from RHS of \eqref{f-lower-upper-bound} with $x=x_{t+1}$, the
second one follows from \eqref{f-tilde-bound}, and the third inequality one from LHS of \eqref{f-lower-upper-bound} and the fact that
$B_g(x;x_{t+1}) \ge \frac{\eta_t}{12}\|x-x_{t+1} \|^3$ with the choice of $g(x) = \frac{\eta_t}{6}\|x-x_t \|^3$ due to Lemma 4 in \citep{Nest08-2}. Hence, \eqnok{inex_rec} follows from definition of $\Delta_{H_t}$ in \eqnok{delta_H}.

\end{proof}

\vgap

In the next result, we establish convergence properties of Algorithm~\ref{alg_INCR} by properly choosing the cubic regularization parameter $\eta_t$ and bounding the error term $\Delta_{H_t}$.
\begin{theorem} \label{theo_INCR}
Let $\{x_t\}_{t \ge 0}$ be generated by Algorithm~\ref{alg_INCR}, Hessian of $f$ be $\gamma$-Lipschitz continuous, and Assumption~\ref{H_assumption} hold. 
Moreover, assume that the cubic regularization parameter is set to
\beq \label{def_eta}
\eta_t = \gamma \ \ \forall t \ge 0.
\eeq
\begin{itemize}
\item [a)] If $f$ is strongly convex with parameter $\lambda>0$, then, for any $t \ge 0$, we have
\beq \label{main_strong}
f(x_{t+1})-f(x_*)  \leq (1-\alpha_t) (f(x_t)-f(x_*)),
\eeq
where
\beq\label{def_alpha_strong}
\alpha_t =\min\left\{\frac{1}{3},\frac{\lambda}{6\mu_t},\sqrt{\frac{2\lambda}{\gamma\|x_t-x_*\|}}\right\}.
\eeq

\item [b)] If $f$ is only convex, $\alpha_t$ is chosen such that $\alpha_0 =1$, and $\alpha_t \in (0,1) \ \ \forall t \ge 1$, then, for any $t \ge 0$, we have
\beq \label{main_cvx}
f(x_{t+1})-f(x_*)
\leq A_t \left[\sum_{i=0}^t \frac{\mu_i \alpha_i^2\|x_i-x_*\|^2}{A_i}+ \gamma \sum_{i=0}^t  \frac{\alpha_i^3 \|x_i-x_*\|^3}{3A_i}\right],
\eeq
where
\beq \label{def_At}
A_{t} :=
\left\{
\begin{array}{ll}
 1, & t = 0,\\
\prod_{i=1}^t(1 - \alpha_{t}), & t \ge 1.
\end{array} \right.
\eeq
\end{itemize}
\end{theorem}

\begin{proof}
We first show part a). Noting \eqnok{H_assump} and \eqnok{delta_H}, we clearly have
\[
\Delta_{H_t} \le \frac12 \langle (H_t-\nabla^2 f(x_t))(x-x_t),x-x_t\rangle \le \frac{\mu_t}{2} \|x-x_t\|^2,
\]
which together with \eqnok{inex_rec} and \eqnok{def_eta} imply that
\[
f(x_{t+1}) \le f(x)+\frac{\mu_t}{2} \|x-x_t\|^2 + \frac{\gamma}{3}\|x-x_t \|^3.
\]
If we take $x= x_t + \alpha (x_* - x_t)$ for some $\alpha \in
(0,1)$, then \eqnok{strong_cvx} implies that
\[f(x) \leq (1-\alpha) f(x_t) + \alpha f(x_*) - \frac{\lambda \alpha
  (1-\alpha)}{2} \|x_*-x_t\|^2.
\]
Combining the above two inequalities, we have
\beqa \label{strongcvx_inex_rec}
f(x_{t+1})-f(x_*)
&\leq& (1-\alpha) (f(x_t)-f(x_*)) - \frac{\lambda \alpha(1-\alpha)}{2}\|x_t-x_*\|^2
+ \frac{\mu_t \alpha^2}{2} \|x_t-x_*\|^2 + \frac{\gamma\alpha^3}{3}\|x_t-x_*\|^3 \nn \\
&=& (1-\alpha) (f(x_t)-f(x_*)) - \frac{\alpha}{2}\|x_t-x_*\|^2 \left[\lambda (1-\alpha)
+ \mu_t \alpha  + \frac{2 \gamma\alpha^2}{3}\|x_t-x_*\| \right].
\eeqa
Taking $\alpha=\alpha_t$ as given in \eqnok{def_alpha_strong}, we have
\[
\lambda (1-\alpha_t) \geq \alpha_t \mu_t + (2\gamma/3) \alpha_t^2
\|x_t-x_*\|,
\]
which together with \eqnok{strongcvx_inex_rec}, clearly imply \eqnok{main_strong}.

We now show part b). Suppose that $f$ is only convex i.e. $\lambda=0$. Hence, setting $\alpha=\alpha_t \in (0,1]$ in \eqnok{strongcvx_inex_rec}, we have
\beq \label{cvx_inex_rec}
f(x_{t+1})-f(x_*)
\leq (1-\alpha_t) (f(x_t)-f(x_*)) + \frac{\mu_t\alpha_t^2}{2} \|x_t-x_*\|^2 + \frac{\gamma\alpha_t^3}{3}\|x_t-x_*\|^3.
\eeq
Dividing both sides of \eqnok{cvx_inex_rec} by $A_t$, setting $\alpha_0=1$, and noting $A_t =(1-\alpha_t)A_{t-1} \ \ t \ge 1$ due to \eqnok{def_At}, we have
\beqa
f(x_1)-f(x_*) &\leq& \frac{\mu_0}{2} \|x_0-x_*\|^2 + \frac{\gamma}{3}\|x_0-x_*\|^3 \nn \\
\frac{f(x_{t+1})-f(x_*)}{A_t}
&\leq& \frac{f(x_t)-f(x_*)}{A_{t-1}} + \frac{\mu_t \alpha_t^2}{2 A_t} \|x_t-x_*\|^2 + \frac{\gamma\alpha_t^3}{3 A_t}\|x_t-x_*\|^3 \ \ \forall t \ge 1.\nn
\eeqa
Summing up both sides of the above inequalities, we obtain \eqnok{main_cvx}.
\end{proof}

\vgap

As it can be seen from the above result, we need boundedness of the iterates generated by Algorithm~\ref{alg_INCR} to provide convergence results. Hence, we make the following assumption.

\begin{assumption} \label{def_R}
Let $\{x_t\}_{t \ge 0}$ be generated by Algorithm~\ref{alg_INCR}. Then, there exists $R>0$ such that
\[
\|x_t - x_*\| \le R \ \ \forall t \ge 0.
\]
\end{assumption}

Setting $\eta_t \ge \gamma$ and $x=x_t$ in Lemma~\ref{main_recur} and noting Assumption~\ref{H_assumption}, we have $f(x_{t+1}) \le f(x_t)$.
Hence, to satisfy Assumption~\ref{def_R}, one only needs to assume that the level set of the starting point is bounded with diameter at most $R>0$ i.e.,
\beq \label{level_bnd}
\cS = \{x: f(x) \le f(x_0) \} \ \ \text{and} \ \ diam(\cS) \le R.
\eeq

In the next result, we specialize rate of convergence of Algorithm~\ref{alg_INCR} when applied to (strongly) convex problems.
\begin{corollary}\label{corl_INCR}
Let $\{x_t\}_{t \ge 0}$ be generated by Algorithm~\ref{alg_INCR}, \eqnok{lips2}, \eqnok{def_eta}, Assumption~\ref{H_assumption}, and \eqnok{level_bnd} hold.
\begin{itemize}
\item [a)] If $f$ is strongly convex with parameter $\lambda>0$, then, for any $t \ge 0$, we have
\beq \label{main_strong2}
f(x_t)-f(x_*)  \leq (1-\bar \alpha)^t (f(x_0)-f(x_*)),
\eeq
where
\beq\label{def_baralpha_strong}
\bar \alpha =\min\left\{\frac{1}{3},\frac{\lambda}{6\mu_u},\sqrt{\frac{2\lambda}{\gamma R}}\right\}.
\eeq

\item [b)] If $f$ is only convex, then, for any $t \ge 0$, we have
\beq \label{main_cvx2}
f(x_t)-f(x_*)
\leq \frac{9 \mu_u R^2}{2(t+2)}+ \frac{3 \gamma R^3}{(t+1)(t+2)}.
\eeq
\end{itemize}
\end{corollary}

\begin{proof}
First, note that under Assumption~\ref{H_assumption}, \eqnok{main_strong2} follows from \eqnok{main_strong} since $\bar \alpha \le \alpha_t \ \ \forall t \ge 0$ due to \eqnok{def_alpha_strong}, \eqnok{def_baralpha_strong}, and \eqnok{level_bnd}.

Second, choosing $\alpha_t$ as
\beq \label{def_alpha_cvx}
\alpha_t = \frac{3}{t+3} \ \ \forall t \ge 0,
\eeq
implies that $\alpha_0 =1$ and for any $t \ge 0$, we have
\begin{align} \label{alpha_A_ineq}
A_t &= \frac{6}{(t+1)(t+2)(t+3)}, \nn \\
\sum_{i=0}^t \frac{\alpha_i^2}{A_i} &= \sum_{i=0}^t \frac{3(i+1)(i+2)}{2(i+3)} \le  \sum_{i=0}^t \frac{3(i+1)}{2} =  \frac{3(t+1)(t+2)}{4}, \nn \\
\sum_{i=0}^t \frac{\alpha_i^3}{A_i} &= \sum_{i=0}^t \frac{9(i+1)(i+2)}{2(i+3)^2} \le  \frac{3t}{2}
\end{align}
due to \eqnok{def_At}. Combining the above relations with \eqnok{main_cvx}, under Assumption~\ref{H_assumption} and \eqnok{level_bnd}, we obtain \eqnok{main_cvx2}.
\end{proof}

\vgap

We now add a few remarks about the above results. First, note that \eqnok{main_strong2} and \eqnok{def_baralpha_strong} imply that the total number of iterations performed by Algorithm~\ref{alg_INCR} to find an $\epsilon$-solution of problem \eqnok{main_prob} is bounded by
\beq \label{strong_bnd}
{\cal O}(1)\max\left\{1,\frac{\mu_u}{\lambda},\sqrt{\frac{\gamma R}{\lambda}}\right\}\ln\left(\frac{f(x_0)-f(x_*)}{\epsilon}\right)
\eeq
when $f$ is strongly convex. While this bound is in the same order of accuracy of the ones obtained by the gradient descent method and the basic ARC method in \citep{CarGouToi12-2} when $f$ has Lipschitz continuous gradient with constant $L$, it would be better than those when $\mu_u \ll L$ and $\lambda \gamma R \ll L^2$. However, the second-order ARC method in \citep{CarGouToi12-2} achieves a better complexity bound with different assumptions on Hessian approximations.
Second, when $f$ is only convex, then \eqnok{main_cvx2} implies that the above complexity bound is changed to
\beq \label{cvx_bnd}
{\cal O}(1)\left\{\frac{\mu_u R^2}{\epsilon}+\sqrt{\frac{\gamma R^3}{\epsilon}}\right\},
\eeq
which is better than those of the gradient descent method and the basic ARC method in \citep{CarGouToi12-2} when $\mu_u \ll L$. Moreover, when exact Hessians are computable i.e., $\mu_u =0$, then the above complexity bounds reduces to the ones obtained in \citep{NestPoly06-1, CarGouToi12-2} and . Finally, note that the above complexity bounds are obtained under Assumption~\ref{H_assumption}. Hence, since this assumption usually holds in expectation or high probability when we use second-order information (and not \eqnok{H_defL}) for Hessian approximations, these complexity bounds should be also viewed accordingly (See Section~\ref{sec_impl}).

In the rest of this section, we want to obtain the above-mentioned complexity results while relaxing Assumption~\ref{H_assumption}. In particular, we show that by adaptively choosing the cubic regularization parameter, this assumption can be relaxed in the following form.
\begin{assumption}\label{H_assumption_mod}
Let $\{x_t\}_{t \ge 0}$ be generated by  Algorithm~\ref{alg_INCR}. Then, we have
\beq \label{H_assump_mod}
0 \preceq H_t, \ \ \|\nabla^2 f(x_t) - H_t \| \le \mu_t \ \ \forall t \ge 0,
\eeq
where $0 \leq \mu_t \leq \mu_u$ for some $\mu_u \ge 0$.
\end{assumption}
Note that in the above assumption, we only want the largest eigenvalue of the error matrix in approximating Hessian to be bounded, while Assumption~\ref{H_assumption}, in addition, requires all eigenvalues of Hessians be upper bounded by those of their approximations.

Next theorem generalize the results of Theorem~\ref{theo_INCR} under the above mild assumption on Hessian approximations.
\begin{theorem} \label{theo_INCR_mod}
Let $\{x_t\}_{t \ge 0}$ be generated by Algorithm~\ref{alg_INCR}, Hessian of $f$ be $\gamma$-Lipschitz continuous, and Assumption~\ref{H_assumption_mod} hold. Moreover, assume that the cubic regularization parameter is adaptively chosen as
\beq \label{def_eta_mod}
\eta_t = \gamma+ 2 \bar \eta_t = \gamma+ \frac{2 \mu_t}{\alpha_t\|x^*-x_t\|} \ \ \forall t \ge 0
\eeq
for some $\alpha_t \in (0,1]$.

\begin{itemize}
\item [a)] If $f$ is $\lambda$-strongly convex and $\alpha_t$ is set to \eqnok{def_alpha_strong}, then we have \eqnok{main_strong}.

\item [b)] If $f$ is only convex, $\alpha_t$ is chosen such that $\alpha_0 =1$, and $\alpha_t \in (0,1) \ \ \forall t \ge 1$, then we obtain
\eqnok{main_cvx}.

\end{itemize}
\end{theorem}

\begin{proof}
Noting \eqnok{delta_H} and \eqnok{H_assump_mod}, we clearly have
\[
\Delta_{H_t} \le \frac{\mu_t}{2} \left[\|x-x_t\|^2+ \|x_{t+1}-x_t\|^2\right],
\]
which together with the choice of $\eta_t$  in \eqnok{def_eta_mod} and \eqnok{inex_rec} imply that
\beqa
f(x_{t+1}) &\le& f(x)+ \frac{\mu_t}{2} \|x-x_t\|^2 + \frac{\gamma+\bar \eta_t}{3}\|x-x_t \|^3 +\frac{\mu_t}{2} \|x_{t+1}-x_t\|^2- \frac{\bar \eta_t}{3} \|x_{t+1}-x_t\|^3 \nn \\
&\le& f(x)+ \frac{\mu_t}{2} \|x-x_t\|^2 + \frac{\gamma+\bar \eta_t}{3}\|x-x_t \|^3 +\frac{\mu_t^3}{6 \bar \eta_t^2},
\eeqa
where the second inequality follows from the fact that $g(r)=\mu r^2 - \eta r^3$ ($\mu, \eta >0$) attains its maximum at $r=2\mu/(3 \eta)$ over the positive orthant. Hence, by choosing $x= x_t + \alpha_t (x_* - x_t)$ for some $\alpha_t \in
(0,1)$ as before and setting $\bar \eta_t = \mu_t/(\alpha_t\|x^*-x_t\|)$ as in \eqnok{def_eta_mod}, we have
\[
f(x_{t+1}) \le (1-\alpha_t) (f(x_t)-f(x_*)) - \frac{\lambda \alpha_t(1-\alpha_t)}{2}\|x_t-x_*\|^2 + \mu_t \alpha_t^2 \|x_*-x_t\|^2 + \frac{\gamma \alpha_t^3}{3}\|x_*-x_t \|^3 .
\]
Rest of the proof is similar to that of Theorem~\ref{theo_INCR} and hence, we skip the details.
\end{proof}

\vgap

Note that implementation of Algorithm~\ref{alg_INCR} under Assumption~\ref{H_assumption_mod} depends on the usually unknown values of $\mu_t$ and $\|x_t-x_*\|$ due to the choice of the cubic regularization parameter in \eqnok{def_eta_mod}. In the next result, we specialize rate of convergence of this algorithm when such a dependence is relaxed under Assumption~\ref{def_R}.
\begin{corollary}
Let the conditions of Theorem~\ref{theo_INCR_mod} for Algorithm~\ref{alg_INCR} and Assumption~\ref{def_R} hold. Moreover, assume that the cubic regularization parameter is set
\beq \label{def_eta_mod2}
\eta_t = \gamma+ \frac{2 \mu_u}{\alpha_t R} \ \ \forall t \ge 0
\eeq
for some $\alpha_t \in (0,1]$.
\begin{itemize}
\item [a)] If $f$ is $\lambda$-strongly convex and $\alpha_t = \bar \alpha \ \ \forall t \ge 0$, where $\bar \alpha$ is given in \eqnok{def_baralpha_strong}, then we have \eqnok{main_strong2}.

\item [b)] If $f$ is only convex and $\alpha_t$ is set to \eqnok{def_alpha_cvx}, then we have \eqnok{main_cvx2}.

\end{itemize}
\end{corollary}

\begin{proof}
The proof is similar to that of Corollary~\ref{corl_INCR} and hence, we skip the details.
\end{proof}

\vgap

It should be mentioned that the first terms in the complexity bounds \eqnok{strong_bnd} and \eqnok{cvx_bnd} for strongly convex and convex problems, respectively, are comparable to the corresponding ones for the gradient descent method. However, the formers are better when $\mu_u$ is sufficiently small. In the next section, we present an acceleration variant of Algorithm~\ref{alg_INCR} which possesses significantly better complexity results than the above-mentioned ones.

\setcounter{equation}{0}
\section{Accelerated Inexact Newton method with cubic regularization} \label{sec_inexact_ac}

Our goal in this section is to present an accelerated variant of Algorithm~\ref{alg_INCR} which is a generalization of the accelerated Newton's method with cubic regularization proposed by \cite{Nest08-2} under inexact second-order information. This algorithm improves the complexity bounds of Algorithm~\ref{alg_INCR} by using the technique of estimate sequence and is formally stated as follows.
\begin{algorithm} [H]
	\caption{The accelerated inexact Newton's method with cubic regularization (AINCR)}
	\label{alg_AINCR}
	\begin{algorithmic}

\STATE Input:
$x_0 \in \bbr^n$, $\eta, \beta >0$, $\alpha_0=1$, $\{\alpha_t\}_{t \ge 1} \in (0,1)$, nonnegative nondecreasing sequence $\{\bar \mu_t\}_{t \ge 0}$, and (strong) convexity parameter $\lambda \ge 0$.
\STATE 0. Compute $x_1= \arg\min_{x \in \bbr^n} \tilde {f}_\gamma (x;x_0)$ and $y_1 =  \arg\min_{x \in \bbr^n}  \{\phi_1(x):= f(x_1)+\frac{\bar \mu_0}{2} \|x-x_0\|^2+\frac{\beta}{6} \|x-x_0\|^3\}$. Set $w_0=x_0$ and $t=1$.
\STATE 1. Set
\beqa
w_t &=& (1 - \alpha_t) x_t + \alpha_t y_t,\label{Ne}\\
x_{t+1} &=& \arg\min_{x \in \bbr^n} \tilde {f}_\eta(x;w_t), \label{def_xt_ac}
\eeqa
where $\tilde {f}(x;w_t)$ is defined in \eqnok{cub_schem}.

\STATE 3. Compute $y_{t+1}$ as
\begin{align} \label{def_yt}
y_{t+1} =  & \arg\min_{x \in \bbr^n} \left\{\phi_{t+1}(x):= \phi_t(x)+\frac{\bar \mu_t-\bar \mu_{t-1}}{2} \|x-x_0\|^2+\frac{\alpha_t}{A_t} \lb(x;x_{t+1}) \right\},
\end{align}
where
\beq \label{def_lb}
\lb(x;x_{t+1})=f(x_{t+1})+\langle \nabla f(x_{t+1}), x-x_{t+1} \rangle +\frac{\lambda}{2}\|x-x_{t+1}\|^2.
\eeq

\STATE 4. Set $t \leftarrow t+1$ and go to step 1.
	\end{algorithmic}
\end{algorithm}

Note that the Hessian approximations in Step 1 of the above algorithm are computed at the generated points $w_t$ and hence in this section, we use the notation of $H_t \approx \nabla^2 f(w_t)$. We also consider the following variant of Assumption~\ref{H_assumption} in this section.

\begin{assumption} \label{H_assumption_ac}
Let $\{w_t\}_{t \ge 0}$ be generated by Algorithm~\ref{alg_AINCR}. Then, we have
\beq\label{H_assump_ac}
 \frac{\mu_t}{2} I \preceq H_t- \nabla^2 f(w_t) \preceq  \mu_t I \ \ \forall t \ge 0,
\eeq
where $0 \le \mu_t \le \mu_u$ for some $\mu \ge 0$.
\end{assumption}
It is worth noting out that disregarding the difference of $x_t$ and $w_t$, Assumption~\ref{H_assumption_ac} clearly implies Assumption~\ref{H_assumption} and it can be also concluded from  Assumption~\ref{H_assumption} by adding a scaled identity matrix to $H_t$ and dividing $\mu_t$ by factor 2. Also, note that Algorithm~\ref{alg_AINCR} differs from the accelerated cubic regularization of Newton's method by \cite{Nest08-2} in using inexact Hessians and introducing the quadratic terms into \eqnok{def_yt}. In particular, even when exact Hessians are used, a quadratic term appears in \eqnok{def_yt}, when $f$ is strongly convex ($\lambda>0$). However, this subproblem with the quadratic terms still has a closed form solution (see Section~\ref{sec_impl}).

We now present a few technical results which will be used to establish convergence properties of Algorithm~\ref{alg_AINCR}. The first one provides an upper bound for function $\phi_t(\cdot)$ computed at each iteration of this algorithm.
\begin{lemma}\label{lem_phi_upper}
Let $\phi_t(\cdot)$ be defined in \eqnok{def_yt}. Then, under Assumption~\ref{H_assumption_ac}, for any $x \in \bbr^n$ and $t \ge 1$, we have
\beq \label{phi_upper2}
\phi_t(x) \le \frac{f(x)}{A_{t-1}} + \frac{\mu_0+\bar \mu_{t-1}}{2} \|x-x_0\|^2+ \frac{2\gamma+\beta}{6} \|x-x_0\|^3.
\eeq
\end{lemma}

\begin{proof}
Noting definition of $x_1$ in the above algorithm, RHS of \eqnok{f-lower-upper-bound} and \eqnok{H_assump_ac}, we have
\beq \label{phi_upper1}
f(x_1) \le \tilde{f}_\gamma(x_1;x_0) = \min_x \tilde{f}_{\gamma}(x;x_0) \le \tilde{f}_{\gamma}(x;x_0) \le f(x)+\frac{\mu_0}{2} \|x-x_0\|^2+ \frac{\gamma}{3} \|x-x_0\|^3,
\eeq
where the last inequality follows from LHS of \eqnok{f-lower-upper-bound}. Also, noting the facts that $\alpha_0=1$, $\sum_{i=0}^t \alpha_i A_i^{-1} = A_t^{-1}$ due to \eqnok{def_At}, \eqnok{def_lb}, and (strong) convexity of $f$, we have
\beqa
\sum_{i=1}^{t-1} \frac{\alpha_i}{A_i} \lb(x;x_{i+1}) &=& \sum_{i=1}^{t-1} \frac{\alpha_i}{A_i} \left[f(x_{i+1})+\langle \nabla f(x_{i+1}), x-x_{i+1}\rangle+\frac{\lambda}{2}\|x-x_{i+1}\|^2  \right] \nn \\
&\le& f(x) \sum_{i=1}^{t-1} \frac{\alpha_i}{A_i} = \left(\frac{1}{A_{t-1}} - \frac{1}{A_0}\right)f(x).\nn
\eeqa
Combining the above relation with \eqnok{def_yt}, \eqnok{phi_upper1}, and noting that $A_0=1$, we obtain \eqnok{phi_upper2}.
\end{proof}

\vgap

We now want to provide a lower bound on $\phi_t(\cdot)$ by evaluating the objective function $f$ at the points $x_t$ generated by Algorithm~\ref{alg_AINCR}. To do so, we first need the a couple of technical results.
\begin{lemma}\label{Berg}
Let $h(x)$ be a strongly convex function with parameter $\bar \lambda \ge 0$, $x_0 \in \bbr^n$, and $\theta_1, \theta_2 \ge 0$. Then,
\[
\bar h(x) \ge \bar h(\bar x)+ \frac{\theta_1+\bar \lambda}{2}\|x-\bar x\|^2 + \frac{\theta_2}{6}\|x-\bar x\|^3,
\]
where $\bar x = \arg\min_{x \in \bbr^n} \left\{\bar h(x):= h(x)+\frac{\theta_1}{2}\|x-x_0\|^2+\frac{\theta_2}{3}\|x-x_0\|^3 \right\}$.
\end{lemma}

\vgap
\begin{proof}
Noting strong convexity of $h$, we have
\beqa
h(x) &\ge& h(\bar x) + \frac{\lambda}{2} \|x-\bar x\|^2 + \langle h'(\bar x), x-\bar x \rangle, \nn \\
\frac{\theta_1}{2}\|x-x_0\|^2 &=& \frac{\theta_1}{2}\|\bar x-x_0\|^2+\frac{\theta_1}{2}\|x-\bar x\|^2 +\theta_1 \langle \bar x - x_0 , x- \bar x \rangle, \nn \\
\frac{\theta_2}{3}\|x-x_0\|^3 &\ge& \frac{\theta_2}{3}\|\bar x-x_0\|^3+\frac{\theta_2}{6}\|x-\bar x\|^3 +\theta_2 \langle \|\bar x - x_0\|(\bar x - x_0), x- \bar x \rangle, \nn
\eeqa
where the last inequality follows from Lemma 4 in \citep{Nest08-2}. We then complete the proof by summing up the above relations and noting that $\langle \bar h'(\bar x), x-\bar x \rangle \ge 0$ due to the convexity of $\bar h(\cdot)$.
\end{proof}

\vgap

In the next result, we provide the relation between input and output of \eqnok{def_xt_ac} in Algorithm~\ref{alg_AINCR}.
\begin{lemma}
Let $\{x_t,w_t\}_{t \ge 1}$ be generated by Algorithm~\ref{alg_AINCR}. Then, under Assumption~\ref{H_assumption_ac}, we have
\beq \label{ineq3}
\nabla f(x_{t+1})^\top (w_t-x_{t+1}) \ge \min \left\{\frac{\|\nabla f(x_{t+1})\|^2}{\sqrt 3 \mu_t}, \sqrt{\frac{\|\nabla f(x_{t+1})\|^3}{4\gamma+2\eta}} \right\}.
\eeq
\end{lemma}
\begin{proof}
For sake of simplicity, we drop the subscripts in this proof. Hence, we can write \eqnok{def_xt_ac} as
\[
x = \arg\min_{x' \in \bbr^n} \tilde{f}(x';w),
\]
which its optimality condition implies that
\[
\nabla f(w) + H (x-w) + \frac{\eta}{2} \|x-w\| (x-w) = 0 .
\]
Letting $r=\|x-w\|$ and noting LHS \eqnok{H_assump_ac}, we have
\begin{align*}
\|\nabla f(x)\| =&
\|\nabla f(w) - \nabla f(x) + H (x-w) + \frac{\eta}{2} \|x-w\| (x-w)\|\\
\leq&
\| \nabla f(x) - \nabla f(w) - \nabla^2 f(w) (x-w)\| + \|(H - \nabla^2
f(w)) (x-w)\|
+ \|\frac{\eta}{2} \|x-w\| (x-w)\|\\
\leq& (\gamma+0.5 \eta) r^2 + \mu r,
\end{align*}
which consequently implies that
\beq \label{r_lower}
r\geq \frac{2\|\nabla f(x)\|}{\mu +
  \sqrt{\mu^2+(4\gamma+2\eta)\|\nabla f(x)\|}}
\eeq
due to nonnegativity of $r$.
Moreover, noting RHS \eqnok{H_assump_ac}, we obtain
\begin{align*}
(\gamma r^2 + 0.5 \mu r)^2 \geq &
(\| \nabla f(x) - \nabla f(w) - \nabla^2 f(w) (x-w)\| + \|(H-\mu I - \nabla^2 f(w)) (x-w)\|^2\\
\geq& \| \nabla f(x) - \nabla f(w) - (H-\mu I) (x-w)\|^2\\
=& \|\nabla f(x)  + (\mu + \frac{\eta}{2} \|x-w\|) (x-w) \|^2 \\
\geq&
\|\nabla f(x)\|^2 + (\mu + \frac{\eta}{2} r)^2 r^2 + (2\mu + \eta r) \nabla f(x)^\top
(x-w).
\end{align*}
Therefore, if $\eta/2 \geq 2 \gamma$, then by re-arranging the terms in the above relation and noting \eqnok{r_lower}, we have
\begin{align}
\nabla f(x)^\top (w-x) \geq &
\frac{1}{2\mu + \eta r} \left[
\|\nabla f(x)\|^2 + 0.75 (\mu + \frac{\eta}{2} r)^2 r^2
\right]
\geq  \sqrt{0.75} \|\nabla f(x)\| r \nn \\
\geq& \frac{ \sqrt{3} \|\nabla f(x)\|^2}
{\mu + \sqrt{\mu^2+(4\gamma+2\eta)\|\nabla f(x)\|}} \ge \left\{
\begin{array}{ll}
 \frac{\|\nabla f(x)\|^2}{\sqrt 3 \mu}, & \text{if} \ \ 3\mu^2 \ge (4\gamma+2\eta)\|\nabla f(x)\|,\\
\sqrt{\frac{\|\nabla f(x)\|^3}{4\gamma+2\eta}}, & \text{otherwise},
\end{array} \right. \nn
\end{align}
which clearly implies \eqnok{ineq3}.
\end{proof}

\vgap

Next simple result shows a lower bound for summation of a power of norm function and its dual.

\begin{lemma}\label{dual}
Let $g(z)= \frac{\theta}{p}\|z\|^p$ for $p \ge 2$ and $g^*$ be its conjugate function i.e., $g^*(v) = \sup_z \{\langle v, z \rangle - g(z) \}$.
Then, we have
\[
g^*(v) = \frac{p-1}{p}\left(\frac{\|v\|^p}{\theta}\right)^{\frac{1}{p-1}}.
\]
Moreover, for any $v,z \in \bbr^n$, we have $g(z)+g^*(v) - \langle z,v \rangle \ge 0$.
\end{lemma}
\begin{proof}
$g^*(v)$ is obtained from definition of $g(z)$ and optimality condition of the $\sup$ in its conjugate form. The second statement also follows from Fenchel's inequality.
\end{proof}

\vgap

We can now provide a lower bound on $\phi_t(\cdot)$ which together with its aforementioned upper bound can guarantee convergence of Algorithm~\ref{alg_AINCR}.

\begin{lemma} \label{lem_phi_lower}
Let $\{x_t,y_t\}_{t \ge 1}$ be generated by Algorithm~\ref{alg_AINCR} and  $\phi_t^*$ be the minimum value of $\phi_t(x)$ defined in \eqnok{def_yt}. Also, assume that the following conditions hold.
\beq\label{conds}
\frac{\alpha_t^2}{A_t} \le \frac{2 \bar \mu_{t-1}+\bar \lambda_t}{\sqrt 3 \mu_t} \ \ \text{and} \ \ \frac{\alpha_t^3}{A_t} \le \frac{9(\beta+3\lambda_t R_t^{-1})}{32(2\gamma+\eta)} \ \ \forall t \ge 1,
\eeq
where
\beq \label{def_lambdat_Rt}
\bar \lambda_t = \frac{\lambda}{2}\left(\frac{1}{A_{t-1}}- 1\right), \ \ \|\bar x_t - y_t\| \le R_t \ \ \forall t \ge 1,
\eeq
and
\beq \label{def_barxt}
\bar x_t = \arg \min_{x \in \bbr^n} \left\{ \frac{\bar \mu_{t-1}+\bar \lambda_t}{2} \|x-y_t\|^2+ \frac{\beta}{12} \| x-y_t\|^3+ \frac{\alpha_t}{A_t} \langle \nabla f(x_{t+1}), x-y_t \rangle \right\}.
\eeq
Then, under Assumption~\ref{H_assumption_ac} for any $x \in \bbr^n$ and $t \ge 1$, we have
\beq \label{phi_lower}
\frac{f(x_t)}{A_{t-1}} \le \phi_t^*.
\eeq
\end{lemma}

\begin{proof}
We use induction to prove \eqnok{phi_lower}. Noting Step 0 of Algorithm~\ref{alg_AINCR}, we have $\phi_1^*=f(x_1)=f(x_1)/A_0$. Now, assume that $f(x_t)/A_{t-1} \le \phi_t^*$ for some $t \ge 1$. Observe that $h(x) := \sum_{i=1}^{t-1} \frac{\alpha_i}{A_i} \lb(x;x_{t+1})$ is strongly convex with parameter $\bar \lambda_t$ due to \eqnok{def_lb} and \eqnok{def_lambdat_Rt}.
Using this choice of $h$ in Lemma~\ref{Berg} with $\bar \lambda = \bar \lambda_t$, $\theta_1 = \bar \mu_{t-1}$, $\theta_2 = \beta/2$, and noting \eqnok{def_yt}, we have
\[
\phi_t \ge \phi_t^* + \frac{\bar \mu_{t-1}+\bar \lambda_t}{2} \|x-y_t\|^2+ \frac{\beta}{12} \|x-y_t\|^3 \ge \frac{f(x_t)}{A_{t-1}}+\frac{\bar \mu_{t-1}+\bar \lambda_t}{2} \|x-y_t\|^2+ \frac{\beta}{12} \|x-y_t\|^3,
\]
which together with monotonicity of $\bar \mu_t$, (strong) convexity of $f$, and \eqnok{def_At} imply that
\begin{align}
& \phi_{t+1}^* = \min_{x \in \bbr^n} \left\{\phi_t(x)+ \frac{\bar \mu_t-\bar \mu_{t-1}}{2} \|x-x_0\|^2+ \frac{\alpha_t}{A_t} \lb(x;x_{t+1}) 
\right\} \nn \\
&\ge \min_{x \in \bbr^n} \left\{\frac{f(x_t)}{A_{t-1}}+\frac{\bar \mu_{t-1}+\bar \lambda_t}{2} \|x-y_t\|^2+ \frac{\beta}{12} \|x-y_t\|^3+ \frac{\alpha_t}{A_t} \lb(x;x_{t+1})
\right\} \nn \\
&\ge \min_{x \in \bbr^n} \left\{\frac{1}{A_{t-1}} \left[f(x_{t+1})+\langle \nabla f(x_{t+1}), x_t-x_{t+1} \rangle +\frac{\lambda}{2} \|x_t-x_{t+1}\|^2 \right]+\frac{\bar \mu_{t-1}+\bar \lambda_t}{2} \|x-y_t\|^2 \right. \nn\\
& \qquad \qquad \qquad \qquad \qquad \qquad \qquad \qquad \qquad \qquad \qquad \qquad \qquad \qquad \quad+ \left.
\frac{\beta}{12} \|x-y_t\|^3 + \frac{\alpha_t}{A_t} \lb(x;x_{t+1})
\right\} \nn \\
&= \min_{x \in \bbr^n} \left\{\frac{1-\alpha_t}{A_t} \left[f(x_{t+1})+\langle \nabla f(x_{t+1}), x_t-x_{t+1} \rangle + \frac{\lambda}{2} \|x_t-x_{t+1}\|^2 \right]+\frac{\bar \mu_{t-1}+\bar \lambda_t}{2} \|x-y_t\|^2 \right. \nn\\
& \qquad \qquad \qquad \qquad \qquad \quad + \left.
\frac{\beta_t}{12} \|x-y_t\|^3 + \frac{\alpha_t}{A_t}
\left[f(x_{t+1})+\langle \nabla f(x_{t+1}), x-x_{t+1} \rangle +\frac{\lambda}{2} \|x-x_{t+1}\|^2\right]
\right\} \nn \\
&= \min_{x \in \bbr^n} \left\{\frac{f(x_{t+1})}{A_t}+\frac{\langle \nabla f(x_{t+1}), w_t-x_{t+1} \rangle}{A_t} +\frac{\bar \mu_{t-1}+\bar \lambda_t}{2} \|x-y_t\|^2+ \frac{\beta_t}{12} \|x-y_t\|^3 \right. \nn\\
& \qquad \qquad \qquad \qquad \qquad \quad + \left. \frac{\alpha_t}{A_t} \langle \nabla f(x_{t+1}), x-y_t \rangle +\frac{\lambda}{2A_t} \left[(1-\alpha_t)\|x_t-x_{t+1}\|^2+ \alpha_t \|x-x_{t+1}\|^2\right]  \right\} \nn \\
&\ge \min_{x \in \bbr^n} \left\{\frac{f(x_{t+1})}{A_t}+\frac{\langle \nabla f(x_{t+1}), w_t-x_{t+1} \rangle}{A_t} +\frac{\bar \mu_{t-1}+\bar \lambda_t}{2} \|x-y_t\|^2+ \frac{\beta}{12} \|x-y_t\|^3+ \frac{\alpha_t}{A_t} \langle \nabla f(x_{t+1}), x-y_t \rangle \right\} \nn \\
&= \frac{f(x_{t+1})}{A_t}+\frac{\langle \nabla f(x_{t+1}), w_t-x_{t+1} \rangle}{A_t} +\frac{\bar \mu_{t-1}+\bar \lambda_t}{2} \|\bar x_t-y_t\|^2+ \frac{\beta}{12} \|\bar x_t-y_t\|^3+ \frac{\alpha_t}{A_t} \langle \nabla f(x_{t+1}), \bar x_t-y_t \rangle, \label{phi_lower2}
\end{align}
where the third and last equality follows from \eqnok{Ne} and \eqnok{def_barxt}, respectively.
To complete the induction, we only need to show nonnegativity of
\beq \label{ineq2}
er_t:= \frac{\langle \nabla f(x_{t+1}), w_t-x_{t+1} \rangle}{A_t}+\frac{\bar \mu_{t-1}+\bar \lambda_t}{2} \|\bar x_t-y_t\|^2 + \frac{\beta}{12} \|\bar x_t-y_t\|^3+ \frac{\alpha_t}{A_t} \langle \nabla f(x_{t+1}), \bar x_t-y_t \rangle.
\eeq
Noting \eqnok{def_barxt}, we have
\[
er_t \ge \frac{\langle \nabla f(x_{t+1}), w_t-x_{t+1} \rangle}{A_t}+\frac{2\bar \mu_{t-1}+\bar \lambda_t}{4} \|\bar x_t-y_t\|^2 + \frac{\beta+3 \bar \lambda_t R_t^{-1}}{12} \|\bar x_t-y_t\|^3+ \frac{\alpha_t}{A_t} \langle \nabla f(x_{t+1}), \bar x_t-y_t \rangle.
\]
%
Letting $z=y_t-x$, $\theta=\mu_{t-1}$, $p=2$, and $v=\frac{\alpha_t}{A_t} \nabla f(x_{t+1})$ in Lemma~\ref{dual}, we have
\beq \label{dual1}
\frac{1}{2 \bar \mu_{t-1}+\bar \lambda_t} \left(\frac{\alpha_t}{A_t}\right)^2\|\nabla f(x_{t+1})\|^2+\frac{2\bar \mu_{t-1}+\bar \lambda_t}{4} \|\bar x_t-y_t\|^2 + \frac{\alpha_t}{A_t} \langle \nabla f(x_{t+1}), \bar x_t-y_t \rangle \ge 0.
\eeq
Similarly, by setting $\theta=\beta/4$ and $p=3$, we obtain
\beq \label{dual2}
\frac{4}{3 \sqrt{\beta+3 \bar \lambda_t R_t^{-1}}} \left(\frac{\alpha_t\|\nabla f(x_{t+1})\|}{A_t}\right)^\frac 32+\frac{\beta+\bar \lambda_t R_t^{-1}}{12} \|\bar x_t-y_t\|^3 + \frac{\alpha_t}{A_t} \langle \nabla f(x_{t+1}), \bar x_t-y_t \rangle \ge 0.
\eeq
Combining \eqnok{ineq3}, \eqnok{dual1}, \eqnok{dual2} under conditions in \eqnok{conds}, we conclude that $er_t$ in \eqnok{ineq2} is nonnegative implying that $\phi^*_{t+1} \ge f(x_{t+1})/A_t$.
\end{proof}

\vgap
We can specify $R_t$ in the above result in more details. We characterize the solution of \eqnok{def_yt} in Section~\ref{sec_impl} and similarly for \eqnok{def_barxt}, we can show that
\beq \label{def_Rt2}
\|y_t - \bar x_t\| \le \frac{\alpha_t A_t^{-1} \|\nabla f(x_{t+1})\|}{\bar \mu_{t-1}+\bar \lambda_t}.
\eeq
Hence, we can choose $R_t$ as an upper bound for the RHS of the above inequality. We are now ready to provide rate of convergence of Algorithm~\ref{alg_AINCR}.

\begin{theorem} Let $\{x_t\}_{t \ge 1}$ be generated by Algorithm~\ref{alg_AINCR} and  Assumption~\ref{H_assumption_ac} hold.

\begin{itemize}
\item [a)] If $f$ is strongly convex with $\lambda>0$ and the parameters are set to
\beq \label{def_mu_strong}
\alpha_t=\hat \alpha = \min \left(\frac89, \left[\frac{\lambda}{18 \sqrt{3}\mu_u}\right]^\frac12, \frac14 \left[\frac{\lambda}{\gamma \bar R}\right]^\frac13\right), \ \ \bar \mu_t = \frac{\lambda}{4} \ \ \forall t \ge 0, \ \ \beta = \frac{3 \lambda}{2 \bar R}, \ \ \text{and} \ \ \eta = 4 \gamma,
\eeq
where $\bar R$ is an upper bound for all $R_t$, then for any $t \ge 1$, we have
\beq \label{main_strongcvx_acc}
f(x_t)-f(x_*) \le 2(1-\hat \alpha)^{t-1}[f(x_0)-(x_*)]\left[\frac{2\mu_u\lambda^{-1}+1}{4}+ \frac{4\gamma \lambda^{-1}+3\bar R^{-1}}{12}\|x_0-x_*\| \right].
\eeq

\item [b)] If $f$ is convex ($\lambda=0$) and the parameters are set to
\beq \label{def_mu_beta_eta}
\alpha_t=\frac{3}{t+3}, \ \ \bar \mu_t = 2 \mu_u(t+2) \ \ \forall t \ge 0, \ \ \beta = 96 \gamma, \ \ \text{and} \ \ \eta = 4 \gamma,
\eeq
then for any $t \ge 1$, we have
\beq \label{main_cvx_acc}
f(x_t)-f(x_*) \le \frac{98 \gamma \|x_0-x_*\|^3}{t(t+1)(t+2)}+\frac{12 \mu_u \|x_0-x_*\|^2}{t(t+2)}.
\eeq
\end{itemize}
\end{theorem}

\begin{proof}
We first show part a). Noting \eqnok{def_lambdat_Rt} and \eqnok{def_mu_strong}, we have
\beqa
\frac{2 \bar \mu_{t-1}+\bar \lambda_t}{\sqrt 3 \mu_t} &\ge& \frac{\lambda (1-\alpha_t)}{2\sqrt{3}\mu_u A_t}, \nn \\
\frac{9(\beta+3\lambda_t R_t^{-1})}{32(2\gamma+\eta)}
&\ge& \frac{9\lambda (1-\alpha_t)}{64 \gamma \bar R A_t},
\eeqa
which imply that the conditions in \eqnok{conds} hold due to the choice of $\alpha_t =\hat \alpha$ in \eqnok{def_mu_strong}. Therefore, combining \eqnok{phi_upper2} and \eqnok{phi_lower}, we have
\[
\frac{f(x_t)}{A_{t-1}} \le \phi_t^* \le \phi_t(x) \le \frac{f(x)}{A_{t-1}} + \frac{\mu_0+\bar \mu_{t-1}}{2} \|x-x_0\|^2+ \frac{2\gamma+\beta}{6} \|x-x_0\|^3.
\]
Letting $x=x_*$ in the above relation and noting \eqnok{def_mu_strong}, \eqnok{def_At}, and the fact that $\|x_0-x_*\|^2 \le 2[f(x_0)-f(x_*)/\lambda$, we obtain \eqnok{main_strongcvx_acc}.

We now show part b). Note that by \eqnok{def_At}, \eqnok{alpha_A_ineq}, and \eqnok{def_mu_beta_eta}, we have
\beqa
\frac{\alpha_t^2}{A_t} &=& \frac{3(t+1)(t+2)}{2(t+3)} \le \frac{4(t+1)}{\sqrt 3} \le \frac{2 \bar \mu_{t-1}}{\sqrt 3 \mu_t}, \nn \\
\frac{\alpha_t^3}{A_t} &=& \frac{9(t+1)(t+2)}{2(t+3)^2} \le \frac{9}{2}
\le \frac{9\beta}{32(2\gamma+\eta)}, \nn
\eeqa
which implies that the conditions in \eqnok{conds} hold since $\bar \lambda_t=0$ due to $\lambda=0$. Noting \eqnok{def_mu_beta_eta}, \eqnok{def_At}, and \eqnok{alpha_A_ineq}, we obtain \eqnok{main_cvx_acc} similar to the last part of part a).

\end{proof}

\vgap

We now add a few remarks about the above result. First, note that \eqnok{main_cvx_acc} implies that the total number of iterations performed by Algorithm~\ref{alg_AINCR} to obtain an $\epsilon$-solution of problem \eqnok{main_prob} when $f$ is strongly convex and after disregarding a logarithmic factor, is bounded by
\beq \label{strong_bnd_ac0}
{\cal O}(1)\max\left\{1,\left(\frac{\mu_u}{\lambda}\right)^\frac12, \left(\frac{\gamma \bar R}{\lambda}\right)^\frac13\right\}\ln\left(\frac{f(x_0)-f(x_*)}{\epsilon}\right).
\eeq
This bound is significantly better than \eqnok{strong_bnd} in terms of dependence on the problem parameters. Second, \eqnok{main_cvx_acc} implies the above complexity bound, when $f$ is convex, is reduced to
\beq \label{cvx_bnd_ac}
{\cal O}(1)\left\{\sqrt{\frac{\mu_u \|x_0-x_*\|^2}{\epsilon}}+\left(\frac{\gamma \|x_0-x_*\|^3}{\epsilon}\right)^\frac13\right\}.
\eeq
This bound obtained by the acceleration scheme is significantly improved in comparison with the one in \eqnok{cvx_bnd}. Moreover, if exact Hessians are available ($\mu_t =\mu =0$), then the above bound is reduced to the one obtain in \citep{Nest08-2}. Third, note that the above bound only depends on the distance between the initial point and the optimal solution. Hence, we do not need to make any assumption on the boundedness of iterates when $f$ is convex.

\vgap

It is worth noting that the bound in \eqnok{strong_bnd_ac0} depends on $\bar R$ which is not clearly specified. Moreover, when the last term inside the max is the dominating one, this bound is not clearly comparable with the optimal bounds for first-order methods. In the rest of this section, we provide a more refined rate of convergence of the acceleration scheme when applied to the strongly convex problems, by properly modifying Algorithm~\ref{alg_AINCR}. More specifically, we restart this algorithm in multiple stages as follows.

\begin{algorithm} [H]
	\caption{The multi-stage AINCR method}
	\label{alg_multi-AINCR}
	\begin{algorithmic}

\STATE Input:
$z_0 \in \bbr^n$, strong convexity parameter $\lambda>0$, Lipschitz constant for Hessians $\gamma>0$, upper bound $\mu_u \ge 0$ on the parameters controlling error in Hessian approximations, and $R_0>0$ such that $\|z_0-x_*\| \le R_0$.

{\bf For $s=1,2,\ldots$:}

{\addtolength{\leftskip}{0.2in}
\STATE 0. Set $x_0=z_{s-1}$ and $R_s = \frac{R_0}{2^s}$.

\STATE 1. Run Algorithm~\ref{alg_AINCR} with $\lambda=0$ and the parameters defined in \eqnok{def_mu_beta_eta} for $T_s$ iterations, where
\beq \label{def_Ts}
T_s = 2\max\left\{\left(\frac{196\gamma R_{s-1}}{\lambda}\right)^\frac13 , 2\left(\frac{6 \mu_u}{\lambda}\right)^\frac12 \right\}.
\eeq
\STATE 2. Set $z_s = x_{T_s}$.\\
}
{\bf End}
	\end{algorithmic}
\end{algorithm}

Note that in the above algorithm, we run Algorithm~\ref{alg_AINCR} for a certain number of iterations and then, we use its output as the initial solution for another run by resetting the parameters and so on. As will be shown in the next result, this re-starting framework allows us to decrease the distance to the optimal solution and the optimality gap in each stage by constant factors.

\begin{theorem}
Let $\{z_s\}_{s \ge 0}$ be generated by Algorithm~\ref{alg_multi-AINCR} and $f$ is strongly convex.
Then, for any $s \ge 0$, we have
\beqa
\|z_s -x_*\|^2 &\le& \left(\frac{1}{2}\right)^s R_0^2, \label{dist_dec} \\
f(z_s) - f(x_*) &\le& \left(\frac{1}{4}\right)^s[f(z_0) - f(x_*)].\label{func_dec}
\eeqa
\end{theorem}

\begin{proof}
We use induction to show \eqnok{dist_dec}. It clearly holds for $s=0$. Assume that it is true for some $s-1 \ge 0$.
Noting strong convexity of $f$, \eqnok{def_mu_beta_eta}, \eqnok{main_cvx_acc}, and \eqnok{def_Ts}, we obtain
\beqa
\|z_s -x_*\|^2 \le \frac{2[f(z_s) - f(x_*)]}{\lambda} \le \frac{2}{\lambda} \left[\frac{98 \gamma \|z_{s-1}-x^*\|^3}{T_s(T_s+1)(T_s+2)}+\frac{12 \mu_u \|z_{s-1}-x^*\|^2}{T_s(T_s+2)}\right] &\le& \frac{\|z_{s-1}-x^*\|^2}{4} \nn \\
&\le& \frac{R_{s-1}^2}{4}= R_s^2,\nn
\eeqa
which completes the induction. Furthermore, by the above relation, we have
\[
f(z_s) - f(x_*) \le \frac{\lambda \|z_{s-1}-x^*\|^2}{8} \le \frac{f(z_{s-1}) - f(x_*)}{4},
\]
which clearly implies \eqnok{func_dec}.
\end{proof}

\vgap

Note that \eqnok{func_dec} implies that Algorithm~\ref{alg_multi-AINCR} can find an $\epsilon$-solution of \eqnok{main_prob} in at most $S=\log_4 \frac{f(z_0) - f(x_*)}{\epsilon}$ stages. Moreover, the total number of iterations performed by this algorithm to find such a solution is bounded by
\[
\sum_{s=1}^S T_s \le 2\left(\frac{196\gamma R_0}{\lambda}\right)^\frac13 \sum_{s=1}^S 2^{\frac{1-s}{3}}+4S\left(\frac{6 \mu_u}{\lambda}\right)^\frac12
\le 8\left(\frac{392\gamma R_0}{\lambda}\right)^\frac13+ 4\sqrt{\frac{6 \mu_u}{\lambda}} \log_4 [f(z_0) - f(x_*)/\epsilon].
\]
Hence, Algorithm~\ref{alg_multi-AINCR} improves the complexity bound in \eqnok{strong_bnd} for strongly convex problems to
\beq \label{strong_bnd_ac}
{\cal O}(1)\left[\sqrt{\frac{\mu_u}{\lambda}} \ln\left(\frac{f(x_0)-f(x_*)}{\epsilon}\right)+\left(\frac{\gamma R_0}{\lambda}\right)^\frac13\right].
\eeq

Note that the second term in the above bound is independent of $\epsilon$. Indeed, if exact Hessians are available ($\mu_u=0$), then the above bound shows the number of required iterations to enter the quadratic convergence region, as shown in \citep{Nest08-2}. Moreover, the above bound is better than \eqnok{strong_bnd_ac0} in terms of both replacing $\bar R$ with $R_0$ and removing the dependence of the last term on the $\epsilon$. Therefore, when gradients of the objective function are $L$-Lipschitz continuous, this complexity bound can be much better than that of optimal first-order methods when $\mu_u \ll L$,

\setcounter{equation}{0}
\section{Implementation issues} \label{sec_impl}
Our goal in this section is to address a few issues regarding implementation of inexact Newton type methods with cubic regularization presented in the previous sections. The first implementation issue is related to solving subproblem \eqnok{def_xt_cub} (and \eqnok{def_xt_ac}) which include cubic regularization terms. \cite{NestPoly06-1} showed that this type of subproblems is equivalent to a convex optimization of one variable even for indefinite matrix $H_t$.  However, solving this convex one-dimensional optimization problem requires eigendecomposition of $H_t$ which can be computationally very expensive for large scale problems. The following result discussed in \citep{NestPoly06-1} and later in \citep{CarGouToi11-1} with more details, characterize the global solution to the aforementioned type of subproblems.
\begin{lemma}
Let $x_{t+1}$ be computed by \eqnok{def_xt_cub} and define $h_t = x_{t+1}-x_t$. Then, we have
\beq \label{subprb_char}
(H_t+\lambda_t I)h_t = - \nabla f(x_t),
\eeq
where $\lambda_t = \frac{\eta_t}{2} \|h_t\|$ and $H_t+\lambda_t I$ is positive semidefinite.
\end{lemma}

Usual approaches to solve \eqnok{subprb_char} require eigendecomposition of $H_t$. On the other hand, similar problem arises in trust-region methods which has been shown to have a much cheaper solution. In particular, the Lanczos method first introduced in \citep{GoLuRoTo99} for solving subproblems of trust-region methods and later in \citep{CarGouToi11-1}, transforms subproblem \eqnok{def_xt_cub} to a much lower dimensional (Krylov) space formed by applying powers of $H_t$ to $\nabla f(x_t)$, in which the eigendecomposition can be done very quickly.

The conjugate gradient (CG) method is a well-known approach for approximately solving linear system of equations in the form of \eqnok{subproblem_quad} which appears in the subproblems of Newton type methods. While this approach is not applicable to the subproblems of Newton's method with cubic regularization,  there has been recent effort in finding approximate solution to this kind of subproblems using iterative algorithms.
\cite{AgZhBuHaMa16} presented an algorithm to find an $\epsilon$-global solution of subproblems in the form of \eqnok{subproblem_cubic} which still requires fast approximation of Hessian matrix inversion and its eigenvectors. \cite{CarDuc16} also analysed the efficiency of the gradient descent (GD) method for solving \eqnok{subproblem_cubic} for general matrix $H$. They showed that GD can find an $\epsilon$ global solution of this kind of problems in ${\cal O}(1/\epsilon \log(1/\epsilon))$ iterations. However, practical performance of these approaches are not clear.

Solving \eqnok{def_yt} is the second implementation issue. Considering its optimality condition, we have
\[
\sum_{\tau=1}^t \frac{\alpha_\tau}{A_\tau} [\nabla f(x_{\tau+1})+\lambda(y_{t+1}-x_{\tau+1})] +\bar \mu_t (y_{t+1}-x_0) +\frac{\beta}{2}\|y_{t+1}-x_0\|(y_{t+1}-x_0) = 0,
\]
which together with the fact that $\sum_{\tau=1}^t = \alpha_\tau A_\tau^{-1} =A_t^{-1}-1 $, imply
\[
y_{t+1}-x_0 = \frac{\sum_{\tau=1}^t \frac{\alpha_\tau}{A_\tau} [\nabla f(x_{\tau+1})+\lambda(x_0-x_{\tau+1})]}{\bar \mu_t+\lambda(A_t^{-1}-1) +\frac{\beta}{2}\|y_{t+1}-x_0\|}.
\]
Denoting $r_t = \|y_{t+1}-x_0\|$ and taking norm of both sides, we obtain $0.5 \beta r_t^2 +[\bar \mu_t+\lambda(A_t^{-1}-1)] r_t - \|v_t\|=0$, where $ v_t = \sum_{\tau=1}^t \frac{\alpha_\tau}{A_\tau} [\nabla f(x_{\tau+1})+\lambda(x_0-x_{\tau+1})]$. Substituting the positive root of this equation into the above relation, we obtain the closed form solution of subproblem \eqnok{def_yt} as
\beq
y_{t+1} = x_0 - \frac{2v_t}{\bar \mu_t+\lambda(A_t^{-1}-1) +\sqrt{[\bar \mu_t+\lambda(A_t^{-1}-1)]^2+ 2\beta \|v_t\|}}.
\eeq

The Third implementation issue is related to computing Hessian approximations $H_t$ satisfying Assumptions~\ref{H_assumption},  \ref{H_assumption_mod}, or \ref{H_assumption_ac}. One can simply choose $H_t$ according to \eqnok{H_defL}. However, this choice will not exploit second-order information of the objective function. To properly address this issue, we need to first specify the structure of the objective function $f$. As it is common for many machine learning problems, assume that the objective function in \eqnok{main_prob} is given as a finite sum of functions i.e.,
\beq\label{finit_sum_prob}
\min_{x \in \bbr^n} \left\{f(x) = \frac{1}{m} \sum_{i=1}^m f_i(x)\right\}
\eeq
for some positive integer $m$. Hence, by randomly choosing a subset $S_H \subset \{1,\ldots,m\}$ according to the uniform distribution,
we can compute subsampled hessian of $f$ as
\beq \label{subsample_hessian}
\hat H(x) = \frac{1}{|S_H|}\sum_{j \in S_H} \nabla^2 f_j(x).
\eeq
The following result provides the relation between size of $S_H$ and quality of the above hessian approximation.
\begin{lemma} \label{mu_bound}
Let $f$ be convex and has Lipschitz continuous gradients with constant $L$. If
\beq \label{H_size}
|S_H| \ge \frac{16 L^2 \ln(2n/\Lambda)}{\delta^2}
\eeq
for some $\delta>0$ and $\Lambda \in (0,1)$, then we have
\beq \label{H_prob}
\Pr\left(|\lambda_i(\hat H(x))-\lambda_i(\nabla^2 f(x))| \le \delta \ \ i=1,\ldots,n\right) \ge 1-\Lambda,
\eeq
where $\hat H$ is given by \eqnok{subsample_hessian}.
Moreover, if each $f_i(x)$ is convex, then the factor $16$ in \eqnok{H_size} is reduced to $4$.
\end{lemma}

\begin{proof}
The probabilistic inequality \eqnok{H_prob} follows from applying a couple random matrix concentration inequalities. One can see \citep{RooMah16-2} for a very similar proof in details.
\end{proof}

\vgap

Now suppose that we approximate Hessian of $f(x)$ by $H(x) = \hat H(x) +\delta I$. Then, the above result implies that Assumption~\ref{H_assumption} holds with $\mu_t = 2\delta$ with at least probability of $1-\Lambda$. Similarly, one can satisfy Assumption~\ref{H_assumption_ac} with $\mu_t = 4\delta$ by choosing $H(x) = \hat H(x) +3\delta I$. On the other hand, Lemma~\ref{mu_bound} implies that if we approximate Hessian by $H(x)=\hat H(x)$, then Assumption~\ref{H_assump_mod} holds with  $\mu_t =\delta$ . The advantage of this approximation is that we do not need to regularize the subsampled Hessian by a scaled identity matrix. However, the choice of $\eta_t$ would be larger and require more parameter estimations due to \eqnok{def_eta_mod2}. To avoid this, one can simply set $R$ to a very large number to still have $\eta_t \approx \gamma$.

It is worth noting that by the above choices for Hessian approximations, an $\epsilon$-solution of problem \eqnok{main_prob} can be found with probability of at least $(1-\Lambda)^{\bar k}$, where $\bar k$ denotes the complexity bounds obtained for Algorithms~\ref{alg_INCR}, \ref{alg_AINCR}, and \ref{alg_multi-AINCR}. Furthermore, the number of required subamples to approximate Hessian in practice is much less than the one in \eqnok{H_size} (see Section~\ref{sec_numerical}).

It should be also mentioned that if problem \eqnok{main_prob} has some constraints, then the results of Section~\ref{sec_inexact} still hold. However, one need to solve the cubic regularized subproblem \eqnok{def_xt_cub} over a set of constraints. Also, results of Section~\ref{sec_inexact_ac} seems not to be easily generalized when the problem is constrained.

\setcounter{equation}{0}
\section{Numerical Experiments} \label{sec_numerical}
In this section, we show performance of our proposed algorithms for solving a binary classification problem with the logistic loss function i.e.,
\beq \label{logistic}
\min_{x \in \bbr^n} f(x) = \frac{1}{m} \sum_{i=1}^m \log\left\{1+\exp(-v^i \langle u^i, x\rangle) \right\}+\frac{\lambda}{2}\|x\|^2,
\eeq
where $\lambda>0$ denotes the regularization parameter. Note that the above function is differentiable, convex and has Lipschitz continuous gradients and Hessians. Hence, it fits in our setting for problem \eqnok{main_prob}.
We assume that each data point $(u^i,v^i)$ is coming from a data set, where $u^i \in \bbr^n$ shows the feature vector and $v^i \in \{-1,1\}$ represents its corresponding binary label. We use two data sets from online library LIBSVM at http://www.csie.ntu.edu.tw/$\sim$cjlin/libsvm which are listed in Table~\ref{dataset}. Each data set is divided into two subsets: training set used for implementation of the algorithms and testing set used for evaluating the quality of the obtained classifier. The regularization parameter is also set to $\lambda = 1/m$.

\begin{table}[h]
\caption{Description of real data sets used for binary classification.}
\label{dataset}
\vskip 0.15in
\begin{center}
\begin{small}
\begin{sc}
\begin{tabular}{lcccr}
\hline
Data set & Dimension & Training size & Testing size \\
\hline
Mushrooms & 112& 6499& 1625 \\
Gisette & 5000& 6000& 1000\\
\hline
\end{tabular}
\end{sc}
\end{small}
\end{center}
\vskip -0.1in
\end{table}

In our first experiment, we implement the INCR method with different number of samples used to compute the subsampled Hessian matrices, namely, $|S_H|=0.001m, 0.005m, 0.025m, 0.125m$ in \eqnok{subsample_hessian}. We use the Lanczos method mentioned in Section~\ref{sec_impl} to solve subproblems of the INCR method up to high accuracy. We report training and testing errors as well as misclassification errors over training and testing data sets. As it can be seen from the summarized results in Figures~\ref{mushrooms_subsample}-\ref{gissette_subsample}, increasing subsample size first slightly improves the performance of the algorithm in terms of number of iterations and then does not change it. However, in terms of the run time, the performance of the algorithm is affected by different subsmaple sizes and the best performance is obtained by choosing $|S_H|=0.005m$. This consequently implies that having more accurate second-order information does not necessarily improve performance of the INCR method. It should be also mentioned that while Lemma~\ref{mu_bound} suggests using larger values of $S_H$ from theoretical point of view, our observation here implies that the choice of $S_H$ in \eqnok{H_size} can be relatively conservative.

\begin{figure}[h]
\begin{center}
\centerline{\includegraphics[width=120mm]{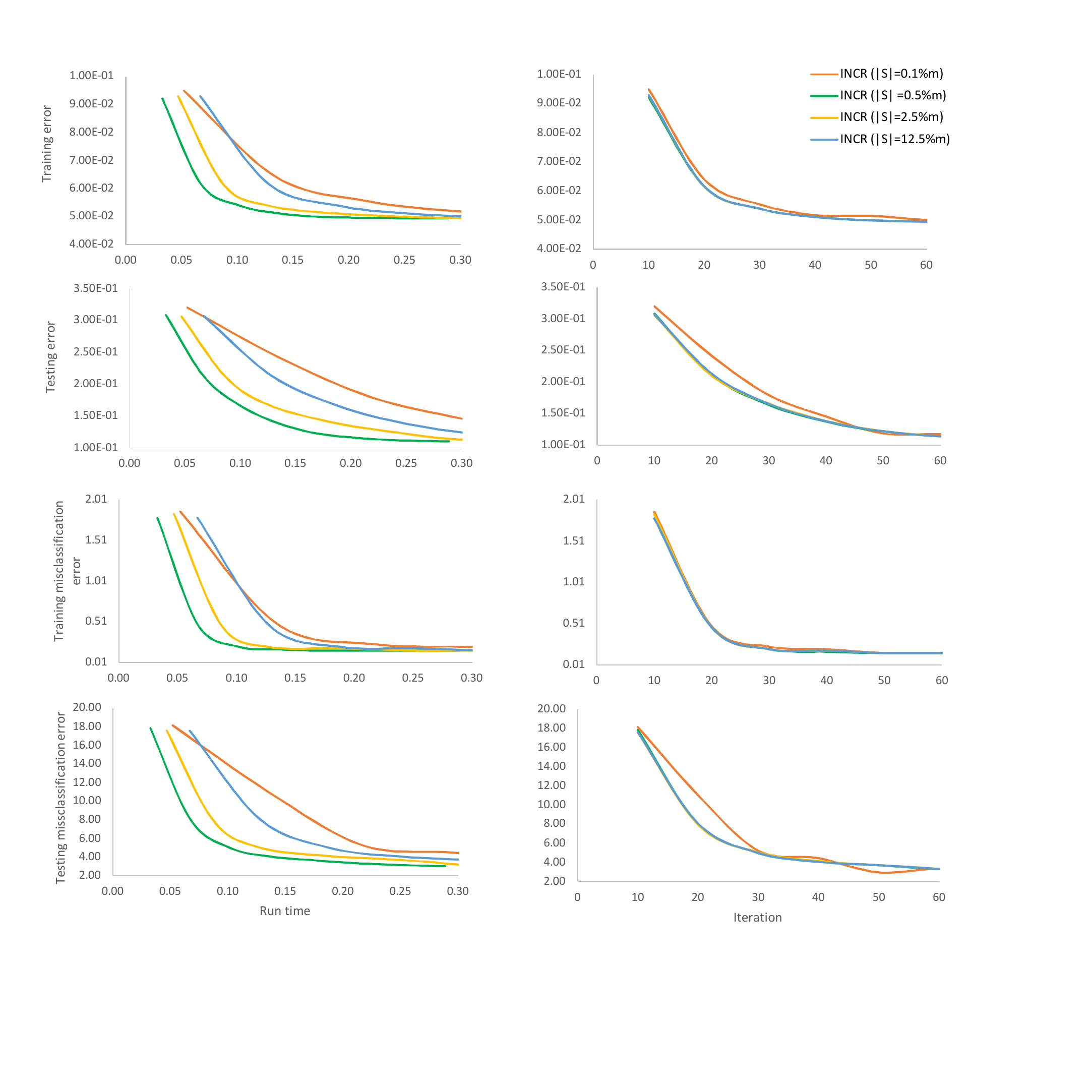}}
\caption{Performance of the INCR method with four different sample sizes using to
compute the subsampled Hessian over Mushrooms data set versus number of iterations on the left and run time on the right.}
\label{mushrooms_subsample}
\end{center}
\end{figure}

\begin{figure}[h]
\begin{center}
\centerline{\includegraphics[width=120mm]{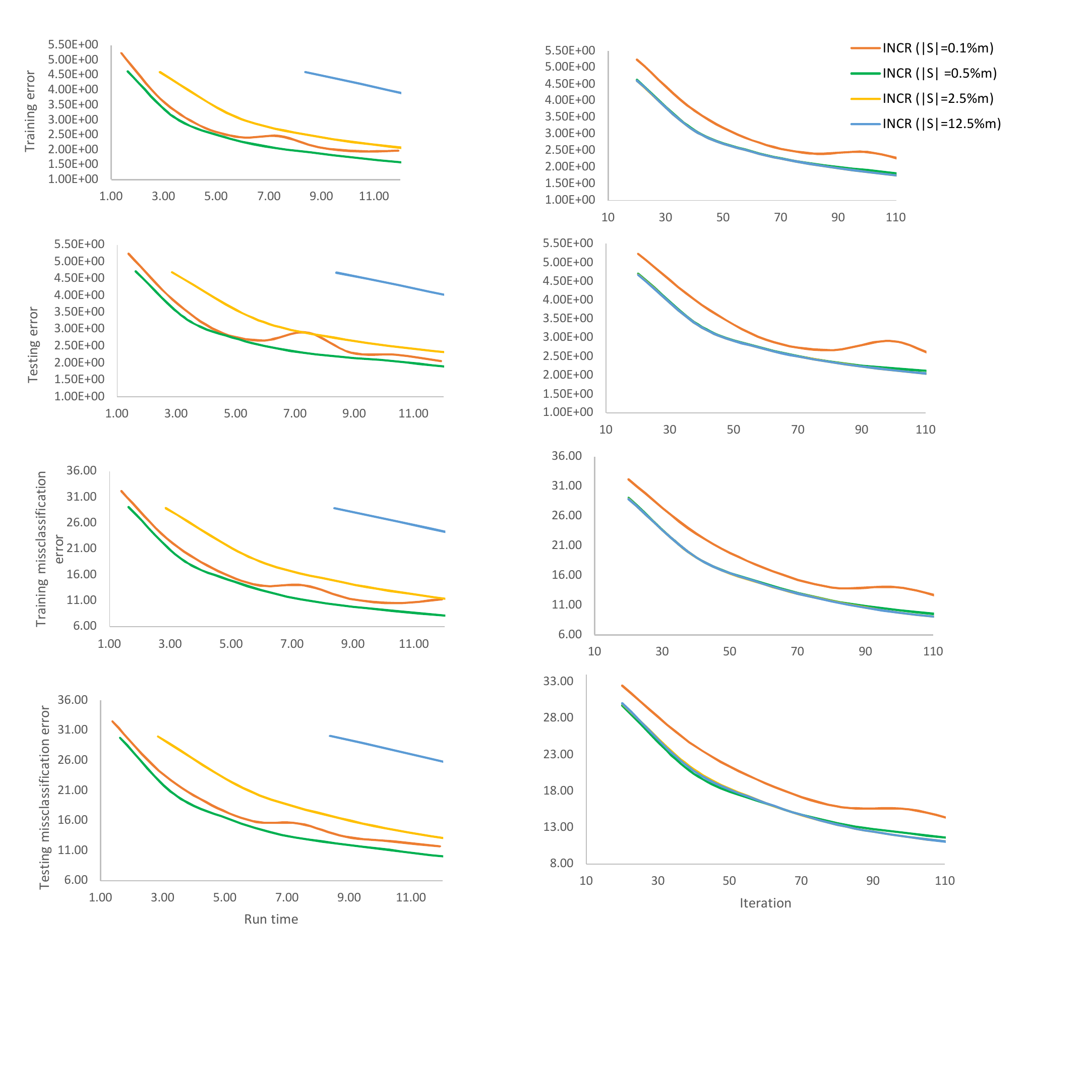}}
\caption{Performance of the INCR method with four different sample sizes using to
compute the subsampled Hessian over Gisette data set versus number of iterations on the left and run time on the right.}
\label{gissette_subsample}
\end{center}
\end{figure}

In the second experiment, we implement the INCR, the AINCR, and the multi-stage AINCR methods with $|S_H|=0.005m$, a variant of Nesterov's accelerated gradient (AG) method in \citep{GhaLan15}, and Algorithm~\ref{alg_INCR} with the choice of $H_t = L I$ (denoted by Cubic-GD) which is a variant of the gradient descent method as mentioned before. Performance of these algorithms are shown in Figures~\ref{mushrooms}-\ref{gisette}.
For both data sets, the AG method performs significantly better than the Cubic-GD method.
The INCR method also performs much better than the AG method in terms of number of iterations and exhibits faster convergence than the AG method in terms of the CPU time. The AINCR method performs better than the INCR method in terms of the number of iteration. However, its run time is comparable or just slightly better. The multi-stage AINCR also performs better than the AINCR method for the Mushrooms data set and their performance are almost comparable over the Gisette data set.

\begin{figure}[h]
\begin{center}
\centerline{\includegraphics[width=120mm]{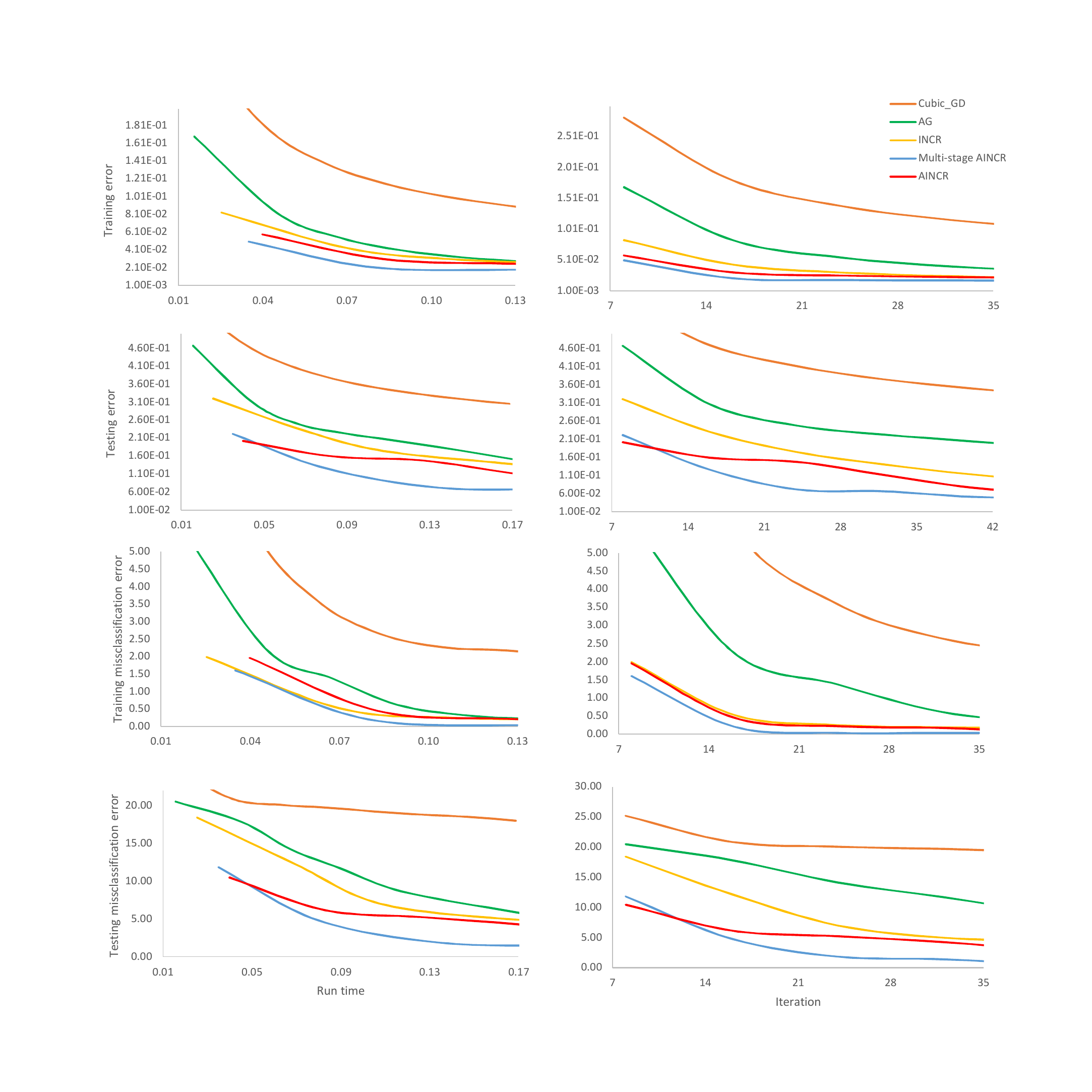}}
\caption{Performance of different algorithms over Mushrooms data set versus number of iterations on the left and run time on the right.}
\label{mushrooms}
\end{center}
\end{figure}

\begin{figure}[h]
\begin{center}
\centerline{\includegraphics[width=120mm]{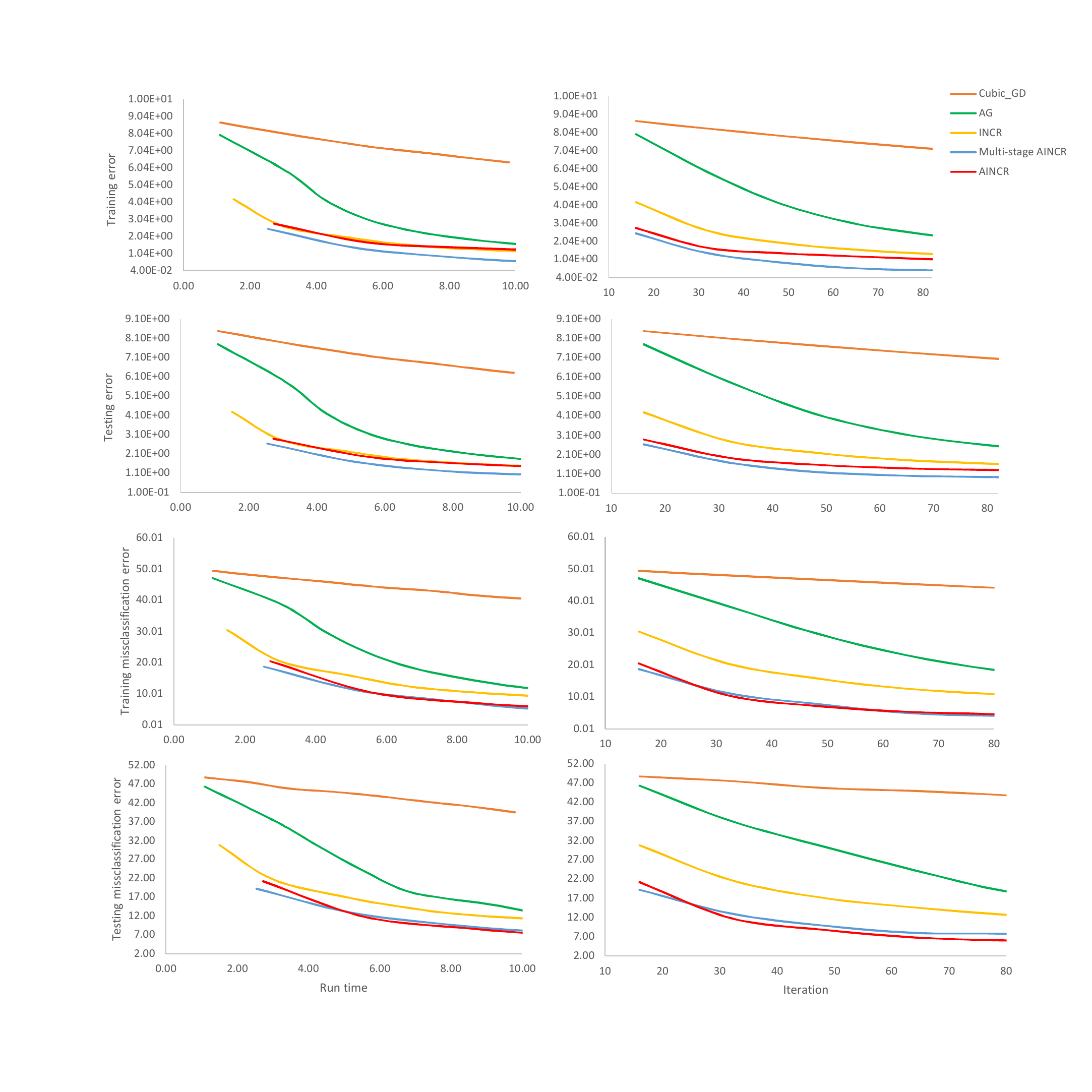}}
\caption{Performance of different algorithms over Gisette data set versus number of iterations on the left and run time on the right.}
\label{gisette}
\end{center}
\end{figure}

\section{Conclusion}
We present an inexact variant of Newton's method with cubic regularization using only Hessian approximations which significantly reduces computational cost of each iteration. Under smoothness assumption on Hessian matrices and mild conditions on their approximations, we provide global rate of convergence of this method when applied to unconstrained (strongly) convex programming. Faster rates of convergence are also established by properly accelerating this method. Our numerical results show superiority of our proposed methods over their exact counterparts and optimal first-order methods.

\acks{We would like to acknowledge support for this project
from the National Science Foundation (NSF grants IIS-1407939 and IIS-1250985). The first author is also thankful to the School of Mathematics at Institute for Research in Fundamental Sciences (IPM) for partially supporting his research}.

\newcommand{\noopsort}[1]{} \newcommand{\printfirst}[2]{#1}
  \newcommand{\singleletter}[1]{#1} \newcommand{\switchargs}[2]{#2#1}

\end{document}